\documentclass[11pt]{amsart}

\usepackage{amssymb,color}
\usepackage{latexsym}
\usepackage{amsmath}
\usepackage{graphicx}
\usepackage{amsthm}
\def\qed{\hfill$\Box$}
\def\proof{\noindent{\it Proof.~~}}
\newtheorem{theorem}{Theorem}
\newtheorem{proposition}{Proposition}

\newtheorem{lemma}{Lemma}
\newtheorem{corollary}{Corollary}

\newtheorem{conjecture}{Conjecture}

\newtheorem{remark}{Remark}

\newcommand{\Z}{\mbox{${\mathbb
 Z}$}}

\title{Rainbow--free $3$--colorings of Abelian Groups}
 \author{Amanda Montejano} 
 \address{Instituto de Matem\'aticas}
 \email{montejano.a@gmail.com}

 \author{Oriol Serra} 
 \address{Departament de Matem\`atica Aplicada IV,
        Universitat Polit\`ecnica de Catalunya.}
 \email{oserra@ma4.upc.edu}

\date{}
\begin{document}

\maketitle

\begin{abstract}
A $3$--coloring of the elements of an abelian group is said to be
rainbow--free if there
is no $3$--term arithmetic progression with its members having
pairwise distinct colors. We give a structural characterization of rainbow--free
colorings of abelian groups. This characterization proves a
conjecture of Jungi\'c et al. on the size of the smallest chromatic
class of a rainbow--free $3$--coloring of cyclic groups.
\end{abstract}

\section{Introduction}

A \emph{$k$--coloring} of a set $X$ is a
map  $c:X\rightarrow [k]$ where $[k]=\{1,2,...k\}$.
 A subset $Y\subset X$ is {\em rainbow} under $c$ if the
coloring assigns pairwise distinct colors to the elements of $Y$.
The study of the existence of rainbow structures falls into the
anti--Ramsey theory initiated by Erd\H{o}s, Simonovits and
S\'os~\cite{erdos}. Arithmetic versions of this theory were
initiated by  Jungi\'{c}, Licht, Mahdian, Ne\v{s}et\v{r}il and
Radoi\v{c}i\'{c}~\cite{jungic} where the authors study the existence
of rainbow arithmetic progressions in colorings of cyclic
groups and of intervals of integers.

\textcolor{black}{In the case of colorings of the
integers, it was shown by Axenovich and Fon der Flaas \cite{axenovich} that every
$3$--coloring of the integer interval $[1,n]$ such that each color class has cardinality at least
$(n+4)/6$ contains a rainbow $3$--term aritmetic progression, thus proving a conjecture stated
in \cite{jungic}. However, the same authors show that, no matter how large is the smaller color class,
there are examples of $k$--colorings with no rainbow arithmetic progressions of length $k\ge 5$. Conlon, Jungi\'{c} and Radoi\v{c}i\'{c}~\cite{conlon} gave a construction
of equinumerous $4$--colorings with no rainbow $4$--term arithmetic progressions. The canonical version of van der Waerden
theorem by Erd\H os and Graham states that every coloring  of the integers (with possibly infinitely many colors) contains
either a monochromatic or a rainbow $k$--term arithmetic progression for each $k$. By the celebrated theorem of Szemer\'edi, if one of the color classes has positive density then one finds a monochromatic arithmetic progression of length $k$. In contrast,
Jungi\'{c} et al.~\cite{jungic} show that there are colorings with all color classes with positive density with no rainbow
$3$--term arithmetic progressions. }

\textcolor{black}{In the above mentioned reference of Jungi\'{c}et al.~\cite{jungic} the authors also study the existence of
rainbow $3$--term arithmetic progressions in $3$--colorings of finite cyclic groups.    The authors characterize all integers $n$ such that every $3$--coloring of the cyclic group $\Z/n\Z$ contains a rainbow $3$-term arithmetic progression.   }

\begin{theorem}[Jungi\'{c} et al.~\cite{jungic}]\label{thm:m=0} For every integer $n$, there
is a rainbow--free $3$--coloring of $\Z/n\Z$ with non--empty color
classes, if and only if $n$ does not satisfy any of the following
conditions:
\textcolor{black}{
\begin{enumerate}
\item[(a)] $n$ is a power of $2$.
\item[(b)] $n$ is a prime and the multiplicative order of $2$ is $n-1$.
\item[(b)] $n$ is a prime, the multiplicative order of $2$ is $(n-1)/2$  and $(n-1)/2$ is odd.
\end{enumerate}
}
\end{theorem}

\textcolor{black}{The above result motivates the following notation. } We denote by ${\mathcal P}_0$   the set
of primes  $p$ for which $2$ has either
multiplicative order $p-1$, or multiplicative order $(p-1)/2$ with
$(p-1)/2$ odd. Let ${\mathcal P}_1$ be the set of remaining primes.

Following \cite{jungic} we let $m(n)$ denote the largest integer $m$
for which there is a rainbow--free $3$--coloring of $\Z/n\Z$ such
that the cardinality of the smallest color class is
$m$. Among other results, the authors
in~\cite{jungic} proved that if the smaller class in a $3$--coloring
of the cyclic group $\Z/n\Z$ has size greater than $n/6$, then there
exists a rainbow $AP(3)$. For $n$ divisible by $6$ this condition
is tight, but for other values of $n$ it is possible to obtain
better bounds.

\begin{theorem}[Jungi\'{c} et al.~\cite{jungic}]\label{thm:n/q} Let $n$ be not a power of
$2$, $q$ be the smallest prime factor of $n$, and $r$ be the
smallest odd prime factor of $n$, then:
$$\left\lfloor \frac{n}{2r} \right\rfloor \le m(n) \le min (\frac{n}{6},
\frac{n}{q})$$
\end{theorem}

\textcolor{black}{Motivated by the above result, Jungi\'c, Ne\v set\v ril and Radoi\v ci\'c~\cite{jnr}
mention that ``Computing the exact value of $m(n)$ remains a challenge'' and} they formulate
 the following conjecture.

\begin{conjecture}[Jungi\'c, Ne\v set\v ril and Radoi\v ci\'c~\cite{jnr}]\label{cjt}
Let $n$ be an integer which is not a power of $2$. Let $p$ denote
the smallest odd prime factor of $n$ in ${\mathcal P}_0$ and let $q$
be the smallest odd prime factor of $n$ in ${\mathcal P}_1$. Then
the largest cardinality of the smallest color class in a
rainbow--free $3$--coloring of the cyclic group $\Z/n\Z$ satisfies
$$m(n)=\left\lfloor \frac{n}{\min \{2p,q\}}\right\rfloor.$$
\end{conjecture}

In this paper we give a  structural \textcolor{black}{characterization} of
$3$--colorings of finite abelian groups of odd order \textcolor{black}{with no rainbow $3$--term arithmetic progressions.} \textcolor{black}{This characterization
provides a proof of Conjecture \ref{cjt} for general abelian groups of odd order. Combined with
the study of the even case, we complete the proof of Conjecture \ref{cjt} for {\it cyclic} groups
of even order as well.} Note  that   Conjecture~\ref{cjt}  does not hold
for general abelian groups of even order as illustrated by the
following counterexample. Let \textcolor{black}{$G:=H\oplus (\Z/2\Z\oplus \Z/2\Z)$}
where $|H|$ is not a power of $2$. Consider the following
$3$--coloring of $G$: let the subgroup $H$ be colored by $A$, color
one of the three remaining $H$--cosets of $G$ by $B$ and the
two cosets left by $C$. This coloring |textcolor{black}{has no rainbow $3$--term arithmetic progressions, since such a
progression (a triple $(x,y,z)$ such that $x-2y+z=0$) must have two of its elements in the
same $H$--coset.}  However the smaller color class has cardinality
$|H|= |G|/4$ which can be arbitrarily larger than $\min
\{\frac{|G|}{2p},\frac{|G|}{q}\}$ according to the choice  of $H$.

Our main result, , Theorem~\ref{thm:main} below,
identifies the three possible kinds of rainbow-free
colorings of an abelian group $G$ of odd order  which can be
described as follows. There is a proper subgroup $H$ of $G$ such
that, either the coloring  of $G$ is obtained by
lifting a rainbow--free coloring with a color class of size one from
 the quotient group $G/H$, or there is one coset of
$H$ which is bichromatic and $G\setminus H$ is monochromatic, or a
combination of the two possibilities above.

\textcolor{black}{In order to state the main result, let us introduce some notation.
Let $G$ be a finite abelian group. Recall that the Minkowski sum of two nonempty
subsets $X,Y\subset G$ is defined as
$$
X+Y=\{ x+y: x\in X, y\in Y\}.
$$
The \emph{period} (or stabilizer) of a subset $S\subseteq G$, denoted by $P (S)$, is
the subgroup of $G$ defined by:}

$$P (S)=\{g\in G:S+g=S\}$$

Thus, $S$ is a union of cosets of $P(S)$, and $P(S)$ lies above any
subgroup of $G$ such that $S$ is a union of cosets. We say that a
set $S$ is $H$--{\it periodic}, where $H$ is a subgroup of $G$, if
$S+H=S$, and $S$ is {\em periodic} if $P(S)$ is a nontrivial
subgroup of $G$ (i.e. $P(S)\neq \{0\}$). \textcolor{black}{ If $P(S)=\{ 0\}$ we say that $S$ is {\it aperiodic}.}

\textcolor{black}{For a  subset $X\subset G$ we denote by $2\cdot X=\{ 2x:\; x\in X\}$ and $-X=\{ -x:\; x\in X\}$.}

\textcolor{black}{A $3$--term arithmetic progression is an ordered triple $(x,y,z)$ with $x,y,z\in G$ satisfying the equation
$x+y=2z$. In the present setting we say that a $3$--coloring of the elements of $G$ is \emph{rainbow--free} if there are no rainbow $3$--term arithmetic progressions. Note that the property of being rainbow--free is
invariant by translations:   $c$ is a
rainbow--free coloring of $G$ if and only if, for each fixed $g\in G$, the
coloring  $c'(x):=c(x+g)$ is also rainbow--free. We will often use
this remark without explicit reference. We  identify a $3$--coloring with the partition of $G$ into
its three color classes which we denote by $\{A, B, C\}$. }

\begin{theorem}\label{thm:main} Let $G$ be a finite abelian
group of odd order $n$ and let $c$ be a $3$--coloring of $G$ with
non--empty color classes $A,B,C$. Then $c$ is rainbow--free if and
only if, up to translation, there is a proper subgroup $H<G$ and a
color class, say $A$, such that \textcolor{black}{the following three conditions hold:}
\begin{enumerate}
\item[{\rm (i)}] $A\subseteq H$, and the $3$--coloring induced in $H$ is
rainbow--free,
\item[{\rm (ii)}] \textcolor{black}{both $\widetilde{B}=B\setminus H$ and $\widetilde{C}=C\setminus H$ are
$H$--periodic sets, and}
\item[{\rm (iii)}] $\widetilde{B}=-\widetilde{B}=2\cdot\widetilde{B}$ and $\widetilde{C}=-\widetilde{C}=2\cdot\widetilde{C}$.
\end{enumerate}
\end{theorem}

The  `if' part of Theorem~\ref{thm:main} can be easily checked and
its proof is detailed  in Section~\ref{sec:proof}, Proposition~\ref{prop:ifmain}.
\textcolor{black}{For the ``only if" part we use results by Kneser \cite{kneser}, Kemperman \cite{kemperman} and
Grynkiewicz \cite{grynkiewicz} which give the structure of sets with small sumset in an abelian group.}

\textcolor{black}{The paper is organized as follows. In Section \ref{sec:prel} we recall some
results in Additive Combinatorics and prove a simple Lemma which will be used in the remainder of the paper.
Section~\ref{sec:small} deals with colorings in which one color class is either small or an arithmetic progression.
In Section \ref{sec:local} we consider colorings with structured color classes. The proof of   Theorem \ref{thm:main}
and the proof of Conjecture \ref{cjt} for abelian groups of odd order is contained in Section~\ref{sec:proof}.  In
Section~\ref{sec:even} we give a structural characterization
for the case of cyclic groups of even order
(Theorem~\ref{thm:mainodd}) which is analogous to Theorem
\ref{thm:main}. With this version   of the characterization one can
complete the proof of Conjecture~\ref{cjt}   for cyclic groups of
even order. }

\section{Some tools from Additive Combinatorics}
\label{sec:prel}

We shall use the following
well--known result of Kneser (see e.g. \cite[Theorem 5.5]{tao})

\begin{theorem}[Kneser]\label{thm:knes}
Let $(A,B)$ be a pair of finite non-empty subsets of an abelian
group $G$. Then, letting $H:=P(A+B)$, we have:
$$|A+B|\ge |A+H|+|B+H|-|H|.$$
Moreover, if $|A+B|\le |A|+|B|-1$  then we have equality.
\end{theorem}

\textcolor{black}{It follows from Kneser`s Theorem that, if
$|A+B|\leq|A|+|B|-1$, then either $A+B$ is periodic or
$|A+B|=|A|+|B|-1$. We shall use this remark in the following sections. }

The structure of pairs of sets $(X,Y)$ in an abelian group $G$
verifying $|X+Y|= |X|+|Y|-1$ is given by the Kemperman Structure
Theorem (KST). We shall only use the following simplified version of
Kemperman's theorem   \cite[Theorem~5.1]{kemperman} (see also the
formulations of Lev \cite[Theorem~2]{lev}, Grynkiewicz
\cite[Theorem KST]{grynkiewicz} \textcolor{black}{and Hamidoune \cite{hamidoune}}).
Let $H\neq\{ 0\}$ be a subgroup of
$G$. A set $S\subset G$ is said to be $H$--{\em quasiperiodic} if it
admits a decomposition $S=S_0\cup S_1$, where each of $S_0$ and
$S_1$ can be empty, $S_1$ is a maximal $H$--periodic  subset of $S$
and $S_0$ is (properly)  contained in a single coset of $H$. Note
that every set $S\subset G$ is quasiperiodic with $S_1=\emptyset$
and $H=G$.

\begin{theorem}[Kemperman \cite{kemperman}]\label{thm:kemp}
Let $A$ and $B$ be nonempty subsets of an abelian group $G$
verifying $$|A+B|=|A|+|B|-1\textcolor{black}{\le |G|-2}.$$ If $A+B$ is
aperiodic then one of the following holds:

\begin{enumerate}
\item[{\rm (i)}] $\min\{ |A|,|B|\}=1$.
\item[{\rm (ii)}] Both $A$ and $B$ are arithmetic progressions
with the same common difference.
\item[{\rm (iii)}] Both $A$ and $B$ are $H$--quasiperiodic for some nontrivial proper
subgroup $H<G$.
\end{enumerate}
\end{theorem}

We shall also use the following extension of KST, recently obtained
by Grynkiewicz~\cite{grynkiewicz}, which describes the structure of
pairs of sets $(X,Y)$ in an abelian group $G$ verifying $|X+Y|=
|X|+|Y|$. Again we only need a simplified version of the full
result.

\begin{theorem}[Grynkiewicz~\cite{grynkiewicz}]\label{thm:gry}
Let $A$ and $B$ be nonempty subsets of an abelian group $G$ of odd
order $n$ verifying $$|A+B|=|A|+|B|\le |G|-3.$$ If $A+B$ is
aperiodic then one of the following holds:

\begin{enumerate}
\item[{\rm (i)}] $\min\{|A|,|B|\}=2$ or $|A|=|B|=3$.
\item[{\rm (ii)}] Both $A$ and $B$ are $H$--quasiperiodic for some nontrivial proper
subgroup $H<G$.
\item[{\rm (iii)}]There are $a, b\in G$ such that $|A'+B'|=|A'|+|B'|-1$
where $A'=A\cup \{ a\}$ and $B'=B\cup \{ b\}$.
\end{enumerate}
\end{theorem}

It is well--known that,  if $A$ and $B$ are subsets of a group $G$
such that $|A|+|B|>|G|$ then $A+B=G$.  The following lemma  characterizes the
structure of sets $A,B\subset G$ with $|A|+|B|=|G|$ and $A+B\neq G$.
We include here a short proof for the benefit of the reader.

\begin{lemma}\label{lem:ph}
Let  $A,B$  be  subsets of a finite
 abelian group $G$.
 \textcolor{black}{
 \begin{enumerate}
\item[{\rm (i)}] If $|A|+|B|>|G|$ then $A+B=G$.
\item[{\rm (ii)}] If $|A|+|B|=|G|$ then either $A+B=G$ or there is a subgroup $H$ and $a\in G$ such that both $A$ and $B$ are $H$--periodic and
 $$A+B=G\setminus (a+H).$$
\end{enumerate}}
 \end{lemma}

\begin{proof} \textcolor{black}{We only prove (ii).} If $|A+B|=|G|-1$ then the statement holds with $H=\{ 0\}$.
Suppose that $|A+B|\le |G|-2$ and let $H=P(A+B)$ \textcolor{black}{be the period of $A+B$}.  By Kneser's
Theorem $$|G|>|A+B|=|A+H|+|B+H|-|H|\ge |A|+|B|-|H|=|G|-|H|.$$ \textcolor{black}{Since $A+B$ is
$H$--periodic, equality holds in the second inequality.} It follows that   $A+B=G\setminus (a+H)$ for some
$a\in G$ and  that $A+H=A$ and $B+H=B$.\qed
\end{proof}

One of the applications of Lemma \ref{lem:ph} is the following
result which will be often used. Let $H$ be a proper subgroup of
$G$. \textcolor{black}{As usual we denote by $G/H$ the quotient
group.}  If $X$ is a subset of $G$ we write $X/H$
for the image of $X$ in $G/H$ by the natural projection
$\pi:G\rightarrow G/H$. We say that a triple $( X,Y,Z)$ of
$H$--cosets is in arithmetic progression if $(X/H)+(Y/H)=2\cdot
(Z/H)$.  For $X$ an $H$--coset and $U$ a color class of a coloring we
write $X_U:=X\cap U$.

\begin{lemma}[The $3$-cosets Lemma]\label{lem:3cosets} Let $\{ A,B,C\}$ be a
rainbow--free $3$--coloring of an abelian group $G$ with odd order
$n$. Let $H<G$ be a subgroup of $G$ and let $(X,Y,Z)$ be a triple of
$H$--cosets in arithmetic progression.

If each of $X_A$, $Y_B$ and $Z_C$ is non--empty, then
\begin{equation}\label{eq:3cosets}
\max\{ |X_A|+|Y_B|, |X_A|+|Z_C|,|Z_C|+|Y_B|\}\leq |H|.
\end{equation}

\textcolor{black}{In particular, none of the three cosets can be monochromatic.}

Moreover, if equality holds then there is a proper subgroup $K<H$
such that two of the sets $X_A, Y_B, Z_C$ are $K$--periodic (the two
involved in the equality holding) and the third one is contained in
a single coset of $K$.
\end{lemma}

\begin{proof}
Since the coloring is rainbow--free \textcolor{black}{and the three cosets are in arithmetic progression} we have $$X_A+Y_B\subseteq
(2\cdot Z)\setminus (2\cdot Z_C).$$ Hence  $|X_A+Y_B|<|H|$ which, \textcolor{black}{by Lemma \ref{lem:ph} (i)},
implies  $|X_A|+|Y_B|\leq |H|$. Similarly $X_A-(2\cdot
Z_C)\varsubsetneq \textcolor{black}{-Y}$ and $Y_B-(2\cdot
Z_C)\varsubsetneq \textcolor{black}{-X}$ imply $|X_A|+|Z_C|\le |H|$
and $|Y_B|+|Z_C|\le |H|$ respectively. This proves the first part of
the statement.

Suppose that $|X_A|+|Y_B|=|H|$. By Lemma~\ref{lem:ph} \textcolor{black}{(ii)} there is a
subgroup $K<H$ such that both $X_A$ and $Y_B$ are $K$--periodic and
$(2\cdot Z)\setminus (X_A+Y_B)$ consists of a single $K$--coset,
which contains $2\cdot Z_C$. A symmetric argument applies if
$|X_A|+|Z_C|=|H|$ or $|Y_B|+|Z_C|=|H|$. \qed
\end{proof}

\section{\textcolor{black}{The prime case}}\label{sec:prime}
\textcolor{black}{
The proof of
Theorem~\ref{thm:main} (in Section ~\ref{sec:proof})
 is by induction on the number of  (not necessarily
distinct) primes dividing $n=|G|$. In this section we prove Theorem~\ref{thm:main}
for groups of prime order, which is a direct consequence of Theorem~\ref{thm:m=0} and Theorem~ref{thm:n/q}.}

\begin{proposition}\label{prop:prime} \textcolor{black}{Let $G$ be a
 group of prime order $p$ and let $c$ be a
$3$--coloring of $G$ with nonempty color classes $A,B,C$. Then $c$
is rainbow--free if and only if, $p\in {\mathcal P}_1$ and, up to
translation, there is a color class, say $A$, such that:}
\begin{enumerate}
\item[{\rm (i)}] \textcolor{black}{$A=\{0\}$,}
\item[{\rm (ii)}] \textcolor{black}{$2\cdot B=B=-B$ and $2\cdot
C=C=-C$.}
\end{enumerate}
\end{proposition}

\begin{proof}
\textcolor{black}{Suppose first that $c$ is rainbow--free. Since
$G\simeq \Z/p\Z$,} it follows from
Theorem~\ref{thm:n/q} that $m(p)\leq 1$ and, from
Theorem~\ref{thm:m=0}, that if $p\in {\mathcal P}_0$ then there are
no rainbow--free $3$--coloring of $\Z/p\Z$ with non--empty color
classes. That is, if $c$ is a rainbow--free
$3$--coloring with nonempty color classes, then necessarily $p\in
{\mathcal P}_1$ and there is a color class, say $A$, such that
$|A|=1$. We can assume that $A=\{0\}$ satisfying (i).

 To prove (ii) note that, for every $x\in B$, since
$G$ is a cyclic group of prime order, then both $-x,2x \in \{B,C\}$.
Hence we must have $-x, 2x\in B$, otherwise we get a rainbow $3$--term arithmetic progression
of the form $(-x,0,x)$ or $(0,x,2x)$. Thus $2\cdot B=B=-B$ and
similarly $2\cdot C=C=-C$.

 Reciprocally, if the coloring satisfies (i) and
(ii), then any $3$--term arithmetic progression containing $0$ has
its remaining members in the same color class, thus $c$ is
rainbow--free.\qed
\end{proof}

\section{Small color classes and color classes in arithmetic progresion}\label{sec:small}

Throughout this section $G$ denotes an abelian
group of odd order $n$ and $c$ is a rainbow--free $3$--coloring of
$G$ with non--empty color classes $\{A,B,C\}$.
\textcolor{black}{The coloring is said to be $H$--{\it regular} if, up to transllation, it satisfies conditions
(i), (ii) and (iii) of Theorem~\ref{thm:main} for the subgroup $H<G$.}

We   \textcolor{black}{begin with   the case} when there is a color
class with just one element.

\begin{lemma}\label{lem:A=1} \textcolor{black}{If
$|A|=1$ then the coloring is $H$--regular  with $H=\{ 0\}$.}
\end{lemma}

\begin{proof}
\textcolor{black}{We may assume that $A=\{0\}$. By choosing $H=\{0\}$,   parts (i) and (ii) of Theorem \ref{thm:main}
 are satisfied.}

To prove (iii) note that,\textcolor{black}{ since $G$ is a group of
odd order,} for every $x\in B$ then both $-x,2x \in \{B,C\}$. Hence we must have $-x,
2x\in B$, otherwise we get a rainbow $3$--term arithmetic progression of the form $(-x,0,x)$
or $(0,x,2x)$. Thus $2\cdot B=B=-B$ and similarly $2\cdot C=C=-C$.
\qed
\end{proof}

Lemma \ref{lem:A=1} \textcolor{black}{provides the description given in
Theorem \ref{thm:main}} when one of the colors has cardinality one.
The $3$-cosets lemma \textcolor{black}{(Lemma~\ref{lem:3cosets})} can
be used to show the analogous \textcolor{black}{result if}  one of the color classes is
contained in a single coset.

\begin{lemma}\label{lem:onecoset}
\textcolor{black}{If $A$ is contained in a single coset of a proper
subgroup $H'<G$ and $|A|>1$, then the coloring is $H$--regular for some proper subgroup  $H<G$.}
\end{lemma}

\begin{proof}
We may assume that \textcolor{black}{$0\in A$. Let $H$ be the  the minimal proper subgroup
of $G$ which contains $A$. } Suppose that $Y\neq H$ is an $H$--coset
which intersects the two remaining color classes $B$ and $C$. Let
$Z$ be a third coset such that $(X=H,Y,Z)$ are in arithmetic
progression. Since $A$ does not meet $Y$ and $Z$, we have
$$|Y_B|+|Y_C|+|Z_B|+|Z_C|=2|H|.$$
It follows from
Lemma~\ref{lem:3cosets} that   $|Y_B|+|Z_C|= |H|$ and $|Y_C|+|Z_B|=
|H|$. Moreover, there is a subgroup $K\le H$ such that $X_A=A$ is
contained in a single coset of $K$ and each of $Y_B, Y_C,Z_B,Z_C$ is
$K$--periodic. By the minimality of $H$ we have $K=H$ contradicting
the existence of the bichromatic coset $Y$. \textcolor{black}{Thus
parts (i) and (ii) of Theorem \ref{thm:main} are satisfied.}

\textcolor{black}{Now consider the $3$--coloring of $G/H$ with color
classes   $A'=\{ 0\}$, $B'=\widetilde{B}/H$ and $
C'=\widetilde{C}/H$ where $\widetilde{B}=B\setminus H$ and
$\widetilde{C}=C\setminus H$. Note that, since the original coloring
$\{A,B,C\}$ of $G$ is rainbow--free, then so it is $\{A',B',C'\}$ in
$G/H$.}

\textcolor{black}{If some of $B'$ or $C'$ is an empty set, then part
(iii) of Theorem~\ref{thm:main} is clearly satisfied.}

\textcolor{black}{If both  $B'$ and $C'$ are nonempty sets, then
$\{A',B',C'\}$ is a rainbow--free coloring of $G/H$ with nonempty
color classes and $|A'|=1$. By Lemma~\ref{lem:A=1}}, it follows that
$2\cdot B'=B'=-B'$ and $2\cdot C'=C'=-C'$. Thus part (iii) is also
satisfied for the coloring $\{A,B,C\}$. \qed\end{proof}

We next handle the cases when  there is a class with
two elements or there are two classes with three elements.

\begin{lemma}\label{lem:A=2} \textcolor{black}{If
$|A|=2$ then the coloring is $H$--regular for some $H<G$.}
\end{lemma}

\begin{proof} By Lemma~\ref{lem:onecoset} we only have to show that
one color is contained in a single coset of a proper subgroup $H<G$.

We may assume that $A=\{0,a\}$. Let us show that $a$ generates a
proper subgroup $H$ of $G$. Suppose on the contrary   that the
cyclic group generated by $a$ is the whole group $G=\langle a
\rangle\cong \Z/n\Z$.

Since $\{-a,a,3a\}$ can not be rainbow, we have $c(-a)=c(3a)$. Since
$\{-3a,0,3a\}$ can not be rainbow we have $c(-3a)=c(3a)$. By
iterating this argument, we have
$$c(-a)=c(3a)=c(-3a)=c(5a)=c(-5a)= ... =c((n-2)a)=c(-(n-2)a),$$
so that the color class of $-a$ has $n-2$ elements. But then the
third one is empty, a contradiction.

Hence \textcolor{black}{$A\subset \langle a \rangle=H<G$} and, by
Lemma~\ref{lem:onecoset}, \textcolor{black}{parts (i), (ii) and (iii)
of} Theorem~\ref{thm:main} \textcolor{black}{are satisfied}.\qed
\end{proof}

\begin{lemma}\label{lem:a=b=3}
 \textcolor{black}{If
$|A|=|B|=3$ and $|A+B|=6$, then the coloring is $H$--regular  for some $H<G$.}
\end{lemma}

\begin{proof} Let $A=\{0,a,a'\}$. It can be shown (see e.g.
\cite{grynkiewicz}) that  only two possibilities occur if
$|A+B|=|A|+|B|=6$: either $B$ is a transllate of $A$ or one of the
sets, say $A$, is an arithmetic progression, and the second one,
$B$, is an arithmetic progression of length four with the same
difference and with one element removed.

Suppose that $B=A+x=\{ x,x+a,x+a'\}$ for some $x\in G$. Since
$\{-x,0,x\}$ can not be rainbow we have $-x\in A\cup B$. If $-x\in
A$  then $0\in A\cap B$, a contradiction. Thus $-x\in B$ and
$a=-2x$. Since $\{0,x,2x\}$ can not be rainbow we have $2x\in A\cup
B$. If $2x\in B$ then $x\in A\cap B$ (since $A=B-x$). Thus $2x=a$
which implies $B=\{ -x,x,3x\}$ and $A=\{-2x,0,2x\}$ and both sets
are arithmetic progressions with the same difference contradicting
$|A+B|=|A|+|B|$.

Suppose now that $A$ is an arithmetic progression with difference
$d$. If $A$ generates a proper subgroup $H$ of $G$ then the result
follows from Lemma \ref{lem:onecoset}. Hence we may assume that
$A=\{0,1,2\}$ and $G$ is the cyclic group of order $n$. Moreover
$B=\{ x, x+2,x+3\}$ for some $x\in \Z/n\Z\setminus
\{0,1,2,n-3,n-2,n-1\}$. If $\{0,x,-x\}$ is not rainbow, since $-x\in
\{1,2\}$ \textcolor{black}{can not hold} and $n$ is odd, we must have $-x=x+3$. But then
$x=(n-3)/2$ and  $\{0,(n-1)/2, (n+1)/2\}$ is rainbow. \textcolor{black}{This contradiction completes the proof.} \qed
\end{proof}

We next \textcolor{black}{consider}  the case when two color classes are \emph{almost}
progressions. An {\it almost}--progression is an arithmetic
progression with one point removed. Observe that, with this
definition, the class of almost progressions contains the class of
all arithmetic progressions except the ones whose length equals the
order of the cyclic group generated by the difference.

\begin{lemma}\label{lem:almost_p}
Assume that $4\le |A|\le |B|\le |C|$. \textcolor{black}{If $A$ and $B$
are almost--progressions with the same difference $d$, then the coloring is $H$--regular  for some $H<G$.}
\end{lemma}

\begin{proof}
If $d$ generates a proper subgroup $H$ of $G$ then $A$ is contained
in a single coset of $H$ and  the result follows by
Lemma~\ref{lem:onecoset}.

Thus we may assume that $d$ generates the full group, so that $G$ is
the cyclic group $\Z/n\Z$ and we may assume $d=1$
\textcolor{black}{(since the property of being  rainbow--free   is invariant by
dilations)}. We will show that in this case $c$ contains a rainbow
$3$--term arithmetic progression.

Let $b$ be the minimum circular distance from elements in $A$ to
elements in $B$.

If $b=1$ we may assume that $n-1\in A$ and $0\in B$. Since $\max\{
|A|,|B|\}\le (n-5)/2$ we have $(n-1)/2\in C$ giving the rainbow $\{
0,(n-1)/2, n-1\}$.

Suppose now that $b>1$. Since $|A|\ge 4$ we may assume that
$\{n-1,0\}\subset A$ and $\{1,2,\ldots ,b'\}\subseteq C$ and
$b'+1\in B$ for some $b'\ge b$. If $b'$ is odd then
$\{0,(b'+1)/2,b'+1\}$ is rainbow and if $b'$ is even then $\{n-1,
(b'+2)/2, b'+1\}$ is rainbow.\qed
\end{proof}

\section{Periodic color classes}\label{sec:local}

In this section we analyze the structure of the color classes when
they are close to be periodic. The consideration of these cases
arise from the discussion on the size of sumsets of the color
classes in a ranbow--free $3$--coloring and the KST and Grynkiewicz
theorems.

Throughout the section \textcolor{black}{we keep the notation of the previous one. Thus $G$
denotes an abelian group of odd order $n$
and $c$ is a rainbow--free $3$--coloring of $G$ with non--empty
color classes $\{A,B,C\}$. The coloring is $H$--regular if it satisfies conditions (i), (ii)
and (iii) of Theorem \ref{thm:main} for some subgroup $H$.}

Recall that, for a subset $X\subset G$ and a subgroup $H<G$, we
denote by $X/H$ the image of $X$ by the natural projection $\pi
:G\rightarrow G/H$.

 \textcolor{black}{The proof of
Theorem~\ref{thm:main} (in Section ~\ref{sec:proof})
 is by induction on the number of  (not necessarily
distinct) primes dividing $n=|G|$, the initial step being proved in Section~\ref{sec:prime}. Thus, in the remainder of this section  we  will assume that
Theorem~\ref{thm:main} holds for any group of order a proper divisor
of $n=|G|$.}

We start with \textcolor{black}{the simplest case.}

\begin{lemma}\label{lem:periodic}
\textcolor{black}{If the three color classes $A,B,$ and $C$ are
$K$--periodic for some subgroup $K<G$, then the coloring is $H$--regular
for some $H<G$.}
\end{lemma}

\begin{proof}
\textcolor{black}{Consider the coloring $A'=A/K$, $B'=B/K$ and
$C'=C/K$ of $G/K$. Note that, since $\{A,B,C\}$ is a rainbow--free
$3$--coloring with nonempty color classes, then so it is
$\{A',B',C'\}$. Since $G/K$ is a group of order a proper divisor of
$n=|G|$, the coloring $\{A',B',C'\}$ is $H'$--regular for some $H'<G/K$ .
In particular, there is a color class, say $A'$, such that
$A'\subseteq H'<G/K$. Thus $A$ is contained in a single coset of the
proper subgroup $H'+K$ in $G$, and the statement follows from Lemma
\ref{lem:onecoset}. \qed }
\end{proof}

We next consider the case where two of the color classes are
quasiperiodic. Recall that a set $S\subset G$ is
$H$--{\em quasiperiodic} if it admits a decomposition $S=S_0\cup
S_1$, where each of $S_0$ and $S_1$ can be empty, $S_1$ is a maximal
$H$--periodic  subset of $S$ and $S_0$ is (properly) contained in a
single coset of $H$.

\begin{lemma}\label{lem:qp}
\textcolor{black}{If $A=A_0\cup A_1$ and $B=B_0\cup B_1$ are
 $K$--quasiperiodic decompositions of $A$ and $B$
with  $K$ a nontrivial proper subgroup of $G$,
 then the coloring is $H$--regular for some $H<G$.}
\end{lemma}

\begin{proof}By Lemma~~\ref{lem:onecoset} we may assume that none of the color classes
is contained in a single coset of a proper subgroup of $G$. If two
of the color classes are periodic then so is the third one and the
result follows from   Lemma~\ref{lem:periodic}. Therefore, up to
renaming the color classes we may assume that each of the sets $A_0,
B_0, A_1$ and $B_1$ are nonempty, and that
$|C/\textcolor{black}{K}|>1$. We also assume that $0\in A_0$.

Let us show that $A_0/\textcolor{black}{K}= B_0/\textcolor{black}{K}$.
Suppose the contrary and let $Z$ be a $\textcolor{black}{K}$--coset
such that $X=\textcolor{black}{K}, Y=B_0+\textcolor{black}{K}$ and $Z$
are in arithmetic progression (such a coset always exists since $n$
is odd). Note that   \textcolor{black}{$X$ intersects $A$ and $C$,    $Y$ intersects $B$ and $C$ and that $Z$ is
monochromatic. This contradicts the $3$--cosets Lemma ( Lemma~\ref{lem:3cosets}).
Hence $A_0/\textcolor{black}{K}= B_0/\textcolor{black}{K}$.}

Consider the $3$--coloring \textcolor{black}{$c_K$} of
$G/\textcolor{black}{K}$ with color classes $\{A',B',C'\}$ where
$A'=A/\textcolor{black}{K}$, $B'=B_1/\textcolor{black}{K}$ and
$C'=G/\textcolor{black}{K}\setminus (A'\cup B')$. Note that
$C'=(C\setminus \textcolor{black}{K})/\textcolor{black}{K}$. Observe
that \textcolor{black}{$c_K$ is a $3$--coloring of $G/K$ with non
empty color classes. Moreover, $c_K$ is  rainbow--free,} otherwise
we have three $\textcolor{black}{K}$--cosets in $G$ in arithmetic
progression where at least two of them are monochromatic (since both
$B_1$ and $C\setminus \textcolor{black}{K}$ are
$\textcolor{black}{K}$--periodic) \textcolor{black}{contradicting}
   the $3$--cosets Lemma (Lemma~\ref{lem:3cosets}).

\textcolor{black}{Since we assume that Theorem~\ref{thm:main} holds in $G/K$,}
there is a \textcolor{black}{proper} subgroup \textcolor{black}{$L<G$} containing
\textcolor{black}{$K$} such that, up to translation, one of the three
chromatic classes of \textcolor{black}{$c_K$} is contained in
\textcolor{black}{$L/K$} and the remaining two  are
\textcolor{black}{$(L/K)$}--periodic outside \textcolor{black}{$L/K$}.

Suppose that $A'\subset \textcolor{black}{L/K}$. Then
\textcolor{black}{$A\subset L$ and the statement follows by Lemma
\ref{lem:onecoset}.}

Let us show now that $C'$ can not be contained in a single coset of
\textcolor{black}{$L/K$} in \textcolor{black}{$G/K$}. Suppose on the
contrary  that $C\textcolor{black}{\setminus K}$ is contained in a
single \textcolor{black}{$L$}--coset $X$ of $G$. Let $Z$ be a
\textcolor{black}{$L$}--coset in arithmetic progression with $X$ and
$Y=A_0+\textcolor{black}{L}$. Since $Y$ intersects the two colors, $A$
and $B$, and $Z$ is necessarily monochromatic with color $A$ or $B$,
we have $|Z|+|Y_A|, |Z|+|Y_B|>\textcolor{black}{|L|}$ contradicting
Lemma~\ref{lem:3cosets}.

Suppose now that $B'$ is contained in a single coset of
\textcolor{black}{$L/K$} in \textcolor{black}{$G/K$. We may assume that
$B_1$ is contained in a single $L$--coset $X\neq L$ in $G$,
otherwise we are done by  Lemma \ref{lem:onecoset}.}

Consider the coloring \textcolor{black}{$c_L$} of
$G/\textcolor{black}{L}$ with color classes $\{A'',B'',C''\}$ where
\textcolor{black}{$A''=(A_1\setminus X)/H,B''=B/\textcolor{black}{L},$}
and $C''=G/\textcolor{black}{L}\setminus (A''\cup B'')$.
\textcolor{black}{Note that $|B''|=2$ and $|C''|>1$, otherwise
$C\subset X$ and $C'$ can not be contained in a single coset of
$L/K$ as shown in the paragraph above.}

\textcolor{black}{If $A''=\emptyset$, consider $Z$ an $L$--coset in
arithmetic progression with $L$ and $X$ (exist since the order of
$G$ is odd). Note that $L$ intersects $A$, $X$ intersects $B$ and
$Z$ is monochromatic of color $C$, a contradiction by
Lemma~\ref{lem:3cosets}.}

\textcolor{black}{If $|A''|>1$, note that $c_L$ is a $3$--coloring of
$G/L$ with non empty color classes. Moreover,} $c_L$ is
rainbow--free since otherwise we have three
\textcolor{black}{$L$}--cosets in $G$ in arithmetic progression where
one of them is monochromatic (since $C''$ is
\textcolor{black}{$L$--periodic) a contradiction by
Lemma~\ref{lem:3cosets}.} \textcolor{black}{Since $|B''|=2$,} it
follows from Lemma~\ref{lem:A=2} that \textcolor{black}{ $B''$ is
contained in a single coset of  a proper subgroup $H/L<G/L$. Thus
$B\subseteq H$ and the statement follows from
Lemma~\ref{lem:onecoset}.}\qed
\end{proof}

We next consider the case where $A+B$ is
\textcolor{black}{$K$}--periodic for some subgroup
\textcolor{black}{$K$} of $G$. Observe that, since $c$ is
rainbow--free, we also have $\textcolor{black}{K}\neq G$.

\begin{lemma}\label{lem:A+Bperiodic}
If $A+B$ is \textcolor{black}{$K$}--periodic for some proper nontrivial subgroup \textcolor{black}{$K$} of
$G$, \textcolor{black}{then the coloring is $H$--regular for some subgroup $H<G$. }
\end{lemma}

\begin{proof} We  show that, under the assumption of the Lemma, both sets $A$
and $B$ admit a \textcolor{black}{$K$}--quasiperiodic decomposition and thus
\textcolor{black}{the statement} follows from Lemma \ref{lem:qp}.

By the Theorem of Kneser we have
\begin{equation}\label{eq1}|A/K+B/K|\ge |A/K|+|B/K|-1.\end{equation}
Since $A+B\cap 2\cdot C=\emptyset$ we have
\begin{equation}\label{eq2}(A+B)/2\subset (G\setminus C)=A\cup B,\end{equation}
where $X/2$ denotes the image of $X\subset G$ by the inverse of the
automorphism of $G$ defined as  $x\mapsto 2x$. This automorphism
leaves all subgroups invariant so that $(A+B)/2$ is also
$K$--periodic. Let $$D=((A\cup B)+K)\setminus (A+B)/2.$$ Note that
the aperiodic parts of $A$ and of $B$ are contained in $D\cup (A\cap
B)$. By (\ref{eq1}) we have
$$
|A/K|+|B/K|-|A/K\cap B/K|=|(A\cup B)/K|=|(A+B)/K|+|D/K|\ge
|A/K|+|B/K|-1+|D/K|,$$
which implies
$$
|D/K|+|A/K\cap B/K|\le 1.
$$
Hence each of $A$ and $B$ admits a $K$--quasiperiodic
decomposition. \qed
\end{proof}

Now we \textcolor{black}{consider} the case where two of the color classes are
\emph{almost} quasiperiodic. A set $X\subset G$ is {\it  almost}
$H$--quasiperiodic (resp. almost $H$--periodic) if there is $x\in G$
such that $X\cup \{ x\}$ is $H$--quasiperiodic (resp.
$H$--periodic).

\begin{lemma}\label{lem:aqp}
If $A$ and $B$ are almost $H$--quasiperiodic for some proper
nontrivial subgroup $H<G$, \textcolor{black}{then the coloring is $H'$--regular for some proper subgroup $H'<G$.}
\end{lemma}

\begin{proof} We say that a coset $X$ of a subgroup $H<G$
 is punctured if all but one of its elements
are in the same color class $U\in \{ A,B,C\}$. We then say that $X$
is a punctured coset of color $U$.

Since $A$ and $B$ are almost $H$--quasiperiodic, they admit
decompositions $A=A_0\cup A_1$ and $B=B_0\cup B_1$  where each of
$A_0$ and $B_0$ are  subsets of some $H$--coset and each of $A_1$
and $B_1$ are almost periodic so that each of them contains at most
one punctured coset.

We may assume that at least one of $A_1$ or $B_1$ contains a
punctured coset and that $0<|A_0|,|B_0|<|H|$ since otherwise $A$ and
$B$ are quasiperiodic and the result follows from
Lemma~\ref{lem:qp}. We may also assume that none of $A,B$ and $C$
are periodic since otherwise at least one of $A+B$ or $A+C$ is
periodic and the result is implied by Lemma~\ref{lem:A+Bperiodic}.
Finally we may assume that $\min \{|A/H|, |B/H|, |C/H|\}>1$ since
otherwise the result is a consequence of Lemma~\ref{lem:onecoset}.

\textcolor{black}{We next show that
the above assumptions lead to a contradiction, thus proving the Lemma.}
We consider two cases:

{\it Case 1\/}: $A_0+H\neq B_0+H$. Let $Z$ be a coset in arithmetic
progression with $X=A_0+H$ and $Y=B_0+H$.

We may assume that one of $X,Y$, say $X$, intersects $C$, since
otherwise $X$ is the punctured coset of $B_1$ and $Y$ is the
punctured coset of $A_1$, which implies that $C$ is periodic. In
particular $X\cap B=\emptyset$. Moreover, whatever the colors
present in $Z$, the conditions of Lemma~\ref{lem:3cosets} are
satisfied and $Z$ can not be a full coset. Since all $H$--cosets
different from $X$ and $Y$ are either monochromatic or punctured,
$Z$ is a punctured coset. Moreover it can not be of color $C$ since
$Z\cap A_0=Z\cap B_0=\emptyset$.

Suppose that $|Z_A|=|Z\cap A|=|H|-1$. Then, again by
Lemma~\ref{lem:3cosets}, $|Z_A|+|X_C|=|H|$, which implies
$|X_C|=|X\cap C|=1$ and $|Y_B|=|Y\cap B|=1$. Thus both $X$ and $Z$
are punctured cosets of color $A$. Since $A$ can not contain more
than two partially filled cosets, $Y$ is a punctured coset of color
$C$. Finally, the other color in $Z$ must also be $C$ since $Z$ is
not the coset containing $B_0$.

Since $|B/H|>1$ there is a coset $Y'\not\in \{X,Y,Z\}$ which
intersects $B$. Moreover, $Y'$ is either a full coset or a punctured
coset of $B$. Let $Z'$ be a third coset in arithmetic progression
with $X$ and $Y'$. Whatever the colors present in $Z'$, the
conditions of Lemma~\ref{lem:3cosets} are satisfied, so that both
$Y'$ and $Z'$ must be  punctured cosets. Thus  $Z'$ must intersect
$C$ (there are no punctured cosets with colors $A$ and $B$) and
$|X_A|+|Y_B|>|H|$, contradicting Lemma~\ref{lem:3cosets}.

Suppose now that $|Z_B|=|Z\cap B|=|H|-1$. If $Y\cap A\neq \emptyset$
then $Y$ is a punctured coset of color $A$ and $|Y_A|+|Z_B|>|H|$
contradicting Lemma~\ref{lem:3cosets}. Otherwise $Y$ intersects $C$
and application of Lemma~\ref{lem:3cosets} implies $|Y_C|=|X_A|=1$.
Thus both $Y$ and $Z$ are punctured cosets of $B$ with second color
$C$ and $X$ is a punctured coset of $C$ with second color $A$, the
same structure as in the case above with colors $A$ and $B$
exchanged, \textcolor{black}{and we again obtain a contradiction with Lemma~\ref{lem:3cosets}.}

{\it Case 2\/}: $A_0+H=B_0+H$. We may assume that at least one of
$A_1$ or $B_1$ contains a punctured coset which is not $X$,
otherwise $A$ and $B$ are quasiperiodic and the results follows from
Lemma~\ref{lem:qp}. So let $Y$ be a punctured coset of color $A$
(observe that $Y_B=\emptyset$ since $B_0$ is contained in $X$, thus
$|Y_C|=1$). Let $Z$ be a coset in arithmetic progression with $X$
and $Y$.

We fist prove that $Z$ is not monochromatic. If $Z$ is monochromatic
of color $B$ (resp. $C$ or $A$) then $|Z_B|+|Y_C|>|H|$ (resp.
$|Z_C|+|Y_A|>|H|$ or $|Z_A|+|Y_C|>|H|$) and we get a contradiction
by Lemma~\ref{lem:3cosets} since $X_A$ (resp. $X_B$) is not empty.

Thus $Z$ must be a punctured coset of color $B$ with $|Z_C|=1$.
Since $|Y_A|+|Z_B|>|H|$ then $X_C=\emptyset$. Since
$|Y_A|+|Z_C|=|H|$ then $|B_0|=1$, but also $|Y_C|+|Z_B|=|H|$ implies
$|A_0|=1$ which is a contradiction. \textcolor{black}{This completes the proof.} \qed
\end{proof}

\section{Proof of Theorem~\ref{thm:main}}\label{sec:proof}

 The next proposition proves the `if' part of
Theorem~\ref{thm:main}.

\begin{proposition}\label{prop:ifmain}  Let $\{ A,B,C\}$ be
a coloring of an abelian group $G$  of odd order. If
there is a proper subgroup $H$ of $G$ and a color class, say $A$,
such that \textcolor{black}{the three following conditions hold:}
\begin{enumerate}
\item[{\rm (i)}]  $A\subseteq H$, and the $3$--coloring induced in $H$ is
rainbow--free,
\item[{\rm (ii)}] \textcolor{black}{both $\widetilde{B}=B\setminus H$ and $\widetilde{C}=C\setminus H$ are
$H$--periodic sets, and}
\item[{\rm (iii)}]  $\widetilde{B}=-\widetilde{B}=2\widetilde{B}$ and
$\widetilde{C}=-\widetilde{C}=2\widetilde{C}$.
\end{enumerate}
\textcolor{black}{Then the $3$-coloring is
rainbow--free.}
\end{proposition}

\begin{proof}
If $C=G\setminus H$ then $A+B$ is contained in $H$,
and thus it is disjoint from $2\cdot C$. Moreover, each of $2\cdot
A$ and $2\cdot B$ are contained in $H$ and thus disjoint from
$A+C=B+C=C$.

Suppose that $C\neq G\setminus H$. Since, \textcolor{black}{by (i)}, a rainbow
$3$--term arithmetic progression in $G$ can not be contained in $H$, it gives rise, by
conditions (ii) and (iii), to a rainbow \textcolor{black}{$3$--term arithmetic progression in $G/H$}  with the
coloring $\{A/H, \widetilde{B}/H, \widetilde{C}/H\}$. However, since
every color class $X$ in this last coloring verifies $X=-X=2\cdot X$
\textcolor{black}{any $3$--term arithmetic progression of $G/H$ } containing $A/H$ has its remaining
members in the same color class. Hence $\{A/H, \widetilde{B}/H,
\widetilde{C}/H\}$ is rainbow--free and so it is $c$. \qed
\end{proof}

\textcolor{black}{It remains} to prove that, if $c$ is a rainbow--free
coloring \textcolor{black}{of an abelian group $G$ of odd order,} then
the color classes verify conditions (i), (ii) and
(iii) \textcolor{black}{of Theorem \ref{thm:main}} with some proper subgroup $H<G$. To prove this we use
the results in sections \ref{sec:prime}, \ref{sec:small} and
\ref{sec:local} together with the theorems of Kneser, Kemperman and
Grynkiewicz.

\begin{proposition}\label{prop:onlyif} \textcolor{black}{Let $\{ A,B,C\}$ be
a rainbow--free coloring of an abelian group $G$  of odd order. }

\textcolor{black}{There is a proper subgroup $H$ of $G$ and a color class, say $A$,
such that  the three following conditions hold:}
\begin{enumerate}
\item[{\rm (i)}]  $A\subseteq H$, and the $3$--coloring induced in $H$ is
rainbow--free,
\item[{\rm (ii)}] \textcolor{black}{both $\widetilde{B}=B\setminus H$ and $\widetilde{C}=C\setminus H$ are
$H$--periodic sets, and}
\item[{\rm (iii)}]  $\widetilde{B}=-\widetilde{B}=2\widetilde{B}$ and
$\widetilde{C}=-\widetilde{C}=2\widetilde{C}$.
\end{enumerate}
\end{proposition}

\begin{proof}
The proof is by induction on the number of  (not
necessarily distinct) primes dividing $n=|G|$. If $n$ is prime, the
statement holds by   Proposition~\ref{prop:prime}.
 \textcolor{black}{We assume} that Theorem~\ref{thm:main} holds
for any group of order a proper divisor of $n=|G|$,
\textcolor{black}{ so that we can use the results in
Section~\ref{sec:local}.}

\textcolor{black}{We first note the following remark.}

\begin{remark}\label{obs:basic}
\textcolor{black}{ For any pair of distinct color classes $X,Y\in
\{A,B,C\}$ we have  $$|X+Y|\le |X|+|Y|.$$ }
\end{remark}

\begin{proof}
\textcolor{black}{Suppose on the contrary that $|X+Y|\ge |X|+|Y|+1$ for some
distinct color classes $X,Y\in \{ A,B,C\}$. Since $n=|G|$ is odd, we have $|2\cdot
Z|=|Z|$ (where $Z$ is the remaining color class). It follows from
the condition $|X+Y|\ge |X|+|Y|+1$ that $(X+Y)\cap (2\cdot Z)\neq \emptyset$,
which implies that there is a rainbow $3$--term arithmetic progression. \qed}
\end{proof}

It follows from the above remark that
\begin{equation}\label{eq:AB}
\textcolor{black}{|A+B|\leq |A|+|B|.}
\end{equation}

By Lemma~\ref{lem:A+Bperiodic} we can assume that $A+B$ is
aperiodic. Then it follows from Kneser's theorem that $|A+B|\ge
|A|+|B|-1$. \textcolor{black}{According to (\ref{eq:AB}) we have to}
consider two cases.

{\it Case 1:} \/ $|A+B|=|A|+|B|-1$. It follows from the simplified
version of the KST,  Theorem~\ref{thm:kemp}, that one of the
following holds:
\begin{enumerate}
\item[{\rm (i)}] $\min\{ |A|,|B|\}=1$. In this case  the result follows by
Lemma~\ref{lem:A=1}.
\item[{\rm (ii)}] Both $A$ and $B$ are arithmetic progressions
with the same common difference $d$.  The result follows by
Lemma~\ref{lem:almost_p}.
\item[{\rm (iii)}] Both $A$ and $B$ are $H$--quasiperiodic for some nontrivial proper
subgroup $H<G$.  The result follows by Lemma~\ref{lem:qp}.
\end{enumerate}

{\it Case 2:} \/ $|A+B|=|A|+|B|$. Then the simplified version of the
Theorem by  Grynkiewicz,  Theorem~\ref{thm:gry}, implies  that one
of the following holds:

\begin{enumerate}
\item[{\rm (i)}]  $\min\{|A|,|B|\}=2$ or $|A|=|B|=3$. In this case the result
follows by Lemmas~\ref{lem:A=2} and~\ref{lem:a=b=3} respectively.
\item[{\rm (ii)}] Both $A$ and $B$ are $H$--quasiperiodic for some nontrivial proper
subgroup $H<G$. The result follows by Lemma~\ref{lem:aqp}.
\item[{\rm (iii)}] There are $a, b\in G$ such that $|A'+B'|=|A'|+|B'|-1$
where $A'=A\cup \{ a\}$ and $B'=B\cup \{ b\}$. According to
Kemperman's Theorem~\ref{thm:kemp}, either $A'+B'$ is periodic, in
which case $A+B$ is also periodic and the result follows by
Lemma~\ref{lem:A+Bperiodic}, or $A'$, $B'$ are both quasiperiodic,
in which case $A$ and $B$ are almost periodic and we can apply
Lemma~\ref{lem:aqp}, or $A'$, $B'$ are both arithmetic progressions
and then $A$ and $B$ are almost arithmetic progressions, a case
handled in Lemma~\ref{lem:almost_p}.
\end{enumerate}
This completes the proof of the Proposition and of Theorem \ref{thm:main}.\qed
\end{proof}

The description of rainbow--free $3$--colorings of abelian groups of
odd order can be used to prove Conjecture~\ref{cjt}.
Actually  Conjecture~\ref{cjt} holds for general abelian
groups of odd order as shown in the next Corollary.

\begin{corollary}\label{cor:cjt} Let $G$ be an abelian group of odd order $n$.
Let $p$ denote the smallest prime factor of $n$ in ${\mathcal P}_0$
and let $q$ be the smallest prime factor of $n$ in ${\mathcal P}_1$.
If $\{A,B,C\}$ is a rainbow--free $3$--coloring of $G$ then
\begin{equation}\label{eq:conjectureodd}\min\{|A|,|B|,|C|\}\leq \left\lfloor \frac{n}{\min
\{2p,q\}}\right\rfloor.\end{equation} Moreover, there are
rainbow--free $3$--colorings of $G$ for which equality holds.
\end{corollary}

\proof We first observe that (\ref{eq:conjectureodd}) is equivalent
to:
\begin{equation}\label{eq:conjectureodd2}\min\{|A|,|B|,|C|\}\leq
\max \left\{\left\lfloor \frac{n}{2p}\right\rfloor,
\frac{n}{q}\right\}.\end{equation}
Note that, since the smallest
prime factor of $n$ is either $p$ or $q$, then the largest proper
subgroup of $G$ has size either $\frac{n}{p}$ or $\frac{n}{q}$.

\textcolor{black}{By
Theorem~\ref{thm:main}~(i), there is a proper subgroup $H<G$ and  one color class, say $A$,   contained
in   $H$. Hence,    $|A|\leq |H|$.}

\textcolor{black}{If $|H|\leq \frac{n}{q}$ then
(\ref{eq:conjectureodd2}) is satisfied.}

\textcolor{black}{ Suppose that  $|H|>
\frac{n}{q}$.  Then necessarily $p<q$ and the size of the largest
proper subgroup of $G$ is $\frac{n}{p}$. Hence $|H|\leq
\frac{n}{p}$.}

\textcolor{black}{ If $|H|<\frac{n}{p}$, since $n$ is an odd number, then
$|H|=\frac{n}{a}$ where $a>2p$. Thus $|H|< \lfloor
\frac{n}{2p}\rfloor$ and  (\ref{eq:conjectureodd2}) is
satisfied.}

\textcolor{black}{Suppose that
$|H|=\frac{n}{p}$.}  Then $G/H$ is a cyclic group of
prime order $p\in {\mathcal P}_0$. By Theorem~\ref{thm:main}~(ii),
each of the two sets $\widetilde{B}=B\setminus H$ and
$\widetilde{C}=C\setminus H$ is a (possibly empty) union of
$H$--cosets. \textcolor{black}{Consider} the
$3$--coloring of $G/H$ with color classes   $A'=\{ 0\}$,
$B'=\widetilde{B}/H$ and $ C'=\widetilde{C}/H$.
\textcolor{black}{Note that, if the original coloring $A,B,C$ of $G$
is rainbow--free, then the induced coloring $A',B',C'$ of $G/H\simeq
\Z/p\Z$ is also rainbow--free. By Proposition~\ref{prop:prime},
since $p\in {\mathcal P}_0$, one of the color classes $B'$ or $C'$
must be empty. Hence either $B/H$ or $C/H$ is an empty set which
implies that}  $G\setminus H$ is monochromatic and
thus $H$ contains two colors. It follows that
$\min\{|A|,|B|,|C|\}\le \lfloor \frac{n}{2p}\rfloor$. \textcolor{black}{This completes the first part of the Corollary.}

Let us show now that there are
rainbow--free $3$--colorings of $G$ for which equality holds in
(\ref{eq:conjectureodd}).

 If $2p\le q$ then choose a subgroup $H<G$ with
cardinality $\frac{n}{p}$. Consider a partition $A\cup B=H$ where
$|A|=\lfloor\frac{n}{2p}\rfloor$, and let $C=G\setminus H$.
 This coloring clearly satisfies parts (i), (ii) and
(iii) of Theorem~\ref{thm:main} and therefore, by Proposition \ref{prop:ifmain}, it is rainbow--free.

 If $q<2p$ then choose a subgroup $H<G$ with
cardinality $\frac{n}{q}$.   We define a coloring
$A,B,C$ of $G$ as follows. Since $q\in {\mathcal P}_1$, by
Theorem~\ref{thm:m=0}, there is a rainbow--free $3$--coloring of
$\Z/q\Z$ with nonempty color classes $A',B'$ and $C'$. Let
$\pi:G\rightarrow G/H\simeq \Z/q\Z$ denote the natural projection
and define $A=\pi^{-1}(A')$, $B=\pi^{-1}(B')$ and $C=\pi^{-1}(C')$.
By Proposition~\ref{prop:prime} this  coloring
satisfies parts (i), (ii) and (iii) of Theorem~\ref{thm:main} and
therefore, by Proposition \ref{prop:ifmain}, it is rainbow--free. \qed

\section{The even case}\label{sec:even}

{\color{black}{In this Section we shall prove Conjecture~\ref{cjt} for {\it cyclic}
groups $\Z/n\Z$ of even order. We recall that, by Theorem \ref{thm:m=0}, if $n=2^m$ then
there are no rainbow--free $3$--colorings of $\Z/n\Z$. Therefore, throughout this Section, we assume that $n=2^ml$ for some $m\ge 1$ and odd $l> 1$. }}

{\color{black}{We start with a Lemma   which will be useful later on.}}

\begin{lemma}\label{lem:Kjust} Let $\{ A,B,C\}$ be a rainbow--free $3$--coloring of {\color{black}{a cyclic}}
group $G$. Suppose that there is a subgroup $H$ such that one of the
colors, say $A$, is contained in $H$, and each of
$\widetilde{B}=B\setminus H$ and $\widetilde{C}=C\setminus H$ are
$H$-periodic.

There is a proper subgroup $K$ of $G$  containing $H$  such that
$$
B+C\supseteq G\setminus K
$$
and each of $B\setminus K$ and $C\setminus K$ is $K$--periodic.
\end{lemma}

\begin{proof}
We consider two cases.

{\it Case 1\/}:  $H=\{ 0\}$. We have $A=\{ 0\}$ \color{black}{and $|B|+|C|=|G|-1$.}

If $|B+C|=|B|+|C|=|G|-1$ then we can choose $K=\{ 0\}$.

Suppose that $|B+C|=|B|+|C|-1$ and let $\{0,x\}=G\setminus (B+C)$.
\color{black}{Since the coloring is rainbow--free we  have $-B=B$ and $-C=C$.
Since $x\not\in B+C$ we have $x-B\cap C=\emptyset$ and $x-C\cap B=\emptyset$.
 Hence}
$$
\{ 0,x\}+B=\{ 0,x\}-B\subset G\setminus C.
$$
{\color{black}{It follows that}}
\begin{equation}\label{eq:b+1}
|\{ 0,x\}+B|\le |B|+1.
\end{equation}
{\color{black}{Let $K=\langle x\rangle$ be the cyclic subgroup generated by $x$ and let
$B=B_1\cup \cdots \cup B_t$
be a decomposition of $B$ into maximal arithmetic progressions with difference $x$.
By the maximality of each $B_i$ we have}} $$B_i+\{ 0,x\}=\min\{|K|, |B_i|+1\}.$$
{\color{black}{By \eqref{eq:b+1}, we see that all but at most  one of the $B_i$'s  are cosets of $K$. Moreover,
if there is one of the $B_i$'s, say $B_1$, which is not a $K$--coset, then
 $B_1$ is an arithmetic progression of difference $x$ properly contained in one $K$--coset.}}

{\color{black}{Likewise $\{0,x\}+C\subset G\setminus B$ implies the analogous structure for $C$. Since none of $B$ and
$C$ contains the whole subgroup $K$,   the only coset where the
proper arithmetic progressions can sit in is $K$ itself. If both colors meet $K$ then we have
a rainbow--free $3$--coloring of this cyclic group with all three colors arithmetic progressions. But we can not partition $K\setminus \{ 0\}$ into two arithmetic progressions $B',C'$ with $B'=-B$ and $C'=-C$. Therefore only one of the two colors meets $K$. This shows that $K$ is a proper subgroup of $G$. Moreover,
$B\setminus K$ and $C\setminus K$ are $K$--periodic.
By the definition of $x$, we also have $B+C\supset G\setminus K$. }}

Finally suppose that $|B+C|< |B|+|C|-1$. By Kneser's theorem  there
is a proper subgroup $K<G$ such that $B+C$ is $K$--periodic and
$|B+C|=|B+K|+|C+K|-|K|$. Since $0\not\in B+C$ we have $B+C\subset
G\setminus K$. Hence,
$$
|G|-1=|B|+|C|\le |B+K|+|C+K|\le |G|.
$$
The last inequalities imply that one of the sets, say $B$, satisfies $B=B+K$ and  the second one
satisfies $|C|=|C+K|-1$. Thus $B$ is $K$--periodic, $C+K=C\cup \{ 0\}$ and $C\setminus K$ is $K$--periodic or empty,
and $B+C=G\setminus K$.

{\it Case 2\/}: $H\neq \{ 0\}$.

 Suppose first that
$\widetilde{C}=G\setminus H$, so that $\widetilde{B}=\emptyset$ and
$B\subset H$. Then $B+C=C$ and the statement holds with $K=H$.

Suppose now that both $\widetilde{B}$ and $\widetilde{C}$ are
nonempty. Then   $\{A'= A/H, B'=\widetilde{B}/H,
C'=\widetilde{C}/H\}$ is a rainbow--free $3$--coloring of $G/H$. It
follows from Case 1 that there is a proper subgroup $K$ of $G$
containing $H$ such that $B'\setminus (K/H)$ and $C'\setminus (K/H)$
are $K$--periodic and $A'+B'\supseteq (G/H)\setminus (K/H)$. It
follows that each of $B\setminus K$ and $C\setminus K$ are
$K$--periodic (one of the two may be empty) and $B+C\supseteq
G\setminus K$. \qed\end{proof}

Let $n=2^ml$ with $l>1$ odd and $m\ge 1$. Then the cyclic group
$G=\Z/n\Z$ can be written as
$$G=L\oplus \Z/2^m\Z$$ where $L$ has odd order $l$ and $m\ge 1$.

As
usual $\{A,B,C\}$ denotes a rainbow--free $3$--coloring of $G$. Let
$P_0=2\cdot G$. Since the even factor of $G$ is cyclic we have
$$P_0\cong L\oplus \Z/2^{m-1}\Z.$$ For each color $X\in \{ A,B,C\}$
we write $X_0=X\cap P_0$ and $X_1=X\cap P_1$ where $P_1$ is the
second  coset of $P_0$ in $G$.

\begin{lemma}\label{lem:monocoset} With the assumptions above, none of the two cosets of $P_0$ is monochromatic.
\end{lemma}

\begin{proof} Suppose the contrary and choose the minimal $m$ for
which there is a counterexample to the statement. We may assume that
$P_1=C_1$. Since $A+B\subset P_0\setminus 2\cdot C_1=P_0\setminus
2\cdot P_1$ we have $m\ge 2$ (otherwise $2\cdot P_1=P_0$) and $A+B$
is contained in the proper subgroup $2\cdot P_0$ of $P_0$. Thus
$A\cup B$ is contained in one coset of $2\cdot P_0$ and  the second
coset of this subgroup in $P_0$  must be colored only with $C$
contradicting the minimality of $m$.\qed
\end{proof}

We next  give the structural result analogous to Theorem
\ref{thm:main} for cyclic groups of even
order.

\begin{theorem}\label{thm:mainodd} Let $G=L\oplus \Z/2^m\Z$
where $L$ has odd order and $m\ge 0$. There is a proper subgroup
$H'$ of $L$ such that one of the colors, say $A$, is contained in
one coset of $H=H'\oplus \Z/2^m\Z$ and each of $B\setminus H$ and
$C\setminus H$ is $H$--periodic.
\end{theorem}

\begin{proof}
By Lemma \ref{lem:monocoset} we may assume that none of the two
cosets of $P_0$ is monochromatic. The proof is by induction on $m$.
By Theorem \ref{thm:main} the result follows for $m=0$. Assume $m\ge
1$.  We consider three cases.

{\it Case 1\/}: One of the two cosets of $P_0$ is bichromatic.

We may assume that $P_1=B_1\cup C_1$ and $0\in A_0=A$. Thus
$$|B_1|+|C_1|=|P_0| \; \mbox{ and }\; B_1+C_1\subset P_0.$$
 \color{black}{Since $c$ is rainbow--free we have  $2\cdot A\subset P_0\setminus (B_1+C_1)$.
 In particular, $B_1+C_1\neq P_0$. It follows from Lemma
\ref{lem:ph}~(ii)    that there
is a proper subgroup $H_1$ of $P_0$ such that $2\cdot A\subseteq
H_1$, and that both $B_1$ and $C_1$ are $H_1$--periodic. Let
$$H_1=H'_1\oplus \Z/2^{m_1}\Z,$$ where $H'_1<L$ has odd order and
$\Z/2^{m_1}\Z$ denotes the cyclic subgroup of $\Z/2^m\Z$ of order $2^{m_1}$ for some $m_1\le m$. We next
consider two cases according to $P_0$ being bichromatic or
trichromatic}.

{\it Case 1.1\/}: $P_0$ is bichromatic. We may assume that
$P_0=A_0\cup B_0$. Since $|A_0|+|B_0|=|P_0|$ and $2\cdot C_1\subset
P_0\setminus (A_0+B_0)$, it follows from Lemma \ref{lem:ph}~(ii)  that
there is a proper subgroup $H_0$ of $P_0$ such that $2\cdot C_1$ is
contained in a single coset of $H_0$, and that both $A_0$ and $B_0$
are $H_0$--periodic. Let $$H_0= H'_0\oplus \Z/2^{m_0}\Z,$$   with
$H'_0<L$ of odd order.

Now $C_1$ being $H_1$--periodic and $2\cdot C_1$ contained in a
single coset of $H_0$ implies $H'_1\le H'_0$. By the analogous
argument on $A_0$ we get $H'_0\le H'_1$. Thus $H'_0=H'_1$ and, by
symmetry, we may assume $H_0\le H_1$. Therefore each color class is
$H_0$--periodic.

It follows that $H'_0$ is a proper subgroup of $L$ since otherwise
we get the rainbow--free $3$--coloring $\{ A/L, B/L, C/L\}$ of the
cyclic group $G/L$ of order $2^m$, contradicting Theorem \ref{thm:m=0}.

Consider the subgroup  $H=H_0'\oplus \Z/2^m\Z$. Observe that $H$
contains  $A_0$ since   $2\cdot A_0\subset H_1\le H$ and the two
subgroups $H_1$ and $H$ have the same odd factor. Similarly, since
$2\cdot C_1\subset H_0$ the color $C_1$ is also contained in $H$.
Hence $B$ does not intersect $H$, since otherwise, as all color
classes are $H'_0$--periodic,  we would get the rainbow-free
$3$--coloring $$\{ (A_0\cap H)/H'_0, (B\cap H)/H'_0, (C\cap H)/H'_0)\}, $$
of the cyclic group $G'/H'_0$ of order $2^m$ contradicting again
Theorem \ref{thm:m=0}. Thus $B=G\setminus H$ and the statement of the
Theorem holds with $H'=H'_0$.

{\it Case 1.2\/}: $P_0$ is trichromatic. By the induction hypothesis
there is a subgroup $$H_0=H'_0\oplus \Z/2^{m-1}\Z,$$ such that
$A\subseteq H_0$ and $B_0\setminus H_0$ and $C_0\setminus H_0$ are
$H_0$--periodic.  Choose a minimal $H'_0$ with this property. We may
assume that $C_0\setminus H_0$ is nonempty. Since $$2\cdot A\subset
H'_1\oplus \Z/2^{m_1}\Z,$$ we have $A\subset H'\oplus \Z/2^{m-1}\Z$
with $H'=H'_0\cap H'_1$. By the minimality of $H'_0$ we have
$H'=H'_0\le H'_1$. Thus,
$$
A\subset H_0.
$$

Suppose that, for some $x\in L\setminus \{ 0\}$, the
coset $X=H_0+(2x,0)$ is colored $B$.  Since  $X\subset A_0+B_0$ is disjoint
from $2\cdot C$,  then the coset $H_0+(x,1)$ is also
monochromatic of  color $B$.

By switching the roles of $B$ and $C$
we conclude that $B_1$ and $C_1$ are also $H_0$--periodic. Let
$$H=H'_0\oplus \Z/2^m\Z.$$

 If $C=G\setminus H$   the statement holds
with $H$ and we are done. Otherwise the $3$--coloring  $$\{ A/H_0,
B/H_0, C/H_0\},$$ is rainbow--free with the color $A'=A/H_0$ consisting only of zero.
Since the coloring is rainbow--free, every color $X$ satisfies $2\cdot
X\subset X\cup \{ 0\}$. This implies that each of $B\setminus H$ and
$C\setminus H$ are not only $H_0$--periodic but in fact
$H$--periodic. This concludes this case.

{\it Case 2\/}: Both cosets of $P_0$ are trichromatic.

By the induction hypothesis there is a  subgroup $$H_0=H'_0\oplus
\Z/2^{m-1}\Z$$ of $P_0$ such that $A_0$ is contained in a single
coset $H_0$ and $B_0\setminus H_0$ and $C_0\setminus H_0$ are
$H_0$--periodic. Choose a minimal $H'_0$ with this property.

It follows from Lemma \ref{lem:Kjust} that there is a proper
subgroup $$T_0=T'_0\oplus \Z/2^{m-1}\Z<P_0$$ containing $H_0$ such
that $B_0+C_0\supset P_0\setminus T_0$ and each of $B_0\setminus
T_0$ and $C_0\setminus T_0$ is $T_0$--periodic.

We have $$2\cdot A_1\subset P_0\setminus (B_0+C_0),$$ so that $2\cdot
A_1\subset T_0$. It follows that  $A_1\subset H$ with $$H=T_0'\oplus
\Z/2^m\Z.$$
 Thus $$A\subset H.$$
  We now use a similar argument to Case
1.2. For each $x\in L\setminus  T_0$ the coset $X=T_0+(2x,0)\subset
P_0\setminus T_0$ is monochromatic and, since the coloring is
rainbow--free, so that $X= A_0+X$ is disjoint from $T_0+(2x,1)$, the
coset $T_0+(x,1)$ is also monochromatic. Hence each of $B_1\setminus
(T_0+A_1)$ and $C_1\setminus (T_0+A_1)$ is also $T_0$--periodic.
Hence, either $G\setminus H$ is monochromatic and we are done, or
$\{ A/H, B/H, C/H\}$ is a rainbow--free $3$--coloring of $G/H$ with
the color class $A'=A/H$ consisting only of zero.
Hence, every color $X$ satisfies $2\cdot X\subset X\cup \{0\}$. This implies that
each of $B\setminus H$ and $C\setminus H$ are not only
$T_0$--periodic but in fact $H$--periodic. This completes the proof.
\qed
\end{proof}

Theorem \ref{thm:mainodd} provides a proof of Conjecture \ref{cjt}.
The proof is completely analogous to the one in Corollary
\ref{cor:cjt} for the case of abelian groups of odd order except
that we invoke Theorem \ref{thm:mainodd} instead of Theorem
\ref{thm:main}.

\begin{corollary} Let $G$ be cyclic group of   order $n$.
Let $p$ denote the smallest odd prime factor of $n$ in ${\mathcal
P}_0$ and let $q$ be the smallest odd prime factor of $n$ in
${\mathcal  P}_1$. If $\{A,B,C\}$ is a rainbow--free $3$--coloring
of $G$ then
\begin{equation}\label{eq:conjodd}\min\{|A|,|B|,|C|\}\leq \left\lfloor \frac{n}{\min
\{2p,q\}}\right\rfloor.\end{equation} Moreover, there are
rainbow--free $3$--colorings of $G$ for which equality holds.
\end{corollary}

\section*{Acknowledgements}

The authors are very grateful to Yahya O. Hamidoune for fruitful and
colorful discussions on the problem considered in this paper.

\end{document}